\documentclass{amsart}

\usepackage{amssymb}
\usepackage{url}

\newtheorem{thm}{Theorem}[section]
\newtheorem{lem}[thm]{Lemma}
\newtheorem{cor}[thm]{Corollary}
\newtheorem{prop}[thm]{Proposition}

\theoremstyle{definition}
\newtheorem{defn}[thm]{Definition}
\newtheorem{example}[thm]{Example}

\theoremstyle{remark}

\newcommand{\BSD}{\mathrm{BSD}}
\newcommand{\Z}{\mathbb{Z}}

\newcommand{\Q}{\mathbb{Q}}
\newcommand{\F}{\mathbb{F}}
\newcommand{\R}{\mathbb{R}}
\newcommand{\C}{\mathbb{C}}
\renewcommand{\O}{\mathcal{O}}
\newcommand{\p}{\mathfrak{p}}
\newcommand{\ran}{r_\mathrm{an}}
\newcommand{\GL}{\mathrm{GL}}
\newcommand{\ord}{\mathrm{ord}}
\newcommand{\Sel}{\mathrm{Sel}}
\newcommand{\rank}{\mathrm{rank}}
\newcommand{\Reg}{\mathrm{Reg}}
\newcommand{\Tr}{\mathrm{Tr}}
\newcommand{\End}{\mathrm{End}}
\newcommand{\Aut}{\mathrm{Aut}}
\newcommand{\tors}{\mathrm{tors}}
\newcommand{\Gal}{\mathrm{Gal}}

\DeclareFontEncoding{OT2}{}{} 
  \newcommand{\textcyr}[1]{%
    {\fontencoding{OT2}\fontfamily{wncyr}\fontseries{m}\fontshape{n}%
     \selectfont #1}}
\newcommand{\Sha}{{\mbox{\textcyr{Sh}}}}

\numberwithin{equation}{section}

\begin{document}

\title[Proving BSD for specific curves of rank 0 and 1]{Proving the Birch and Swinnerton-Dyer \\ conjecture for specific elliptic curves \\ of analytic rank zero and one}
\author[R.L. Miller]{Robert L. Miller}
\address{Mathematics Institute, University of Warwick \\ Coventry CV4 7AL UK}
\thanks{The author was supported in part by NSF DMS Grants \#0354131, \#0757627, \#61-5655, \#61-5801 and \#61-7586.}
\date{} 


\begin{abstract}
We describe an algorithm to prove the Birch and Swinnerton-Dyer conjectural formula for any given elliptic curve defined over the rational numbers of analytic rank zero or one. With computer assistance we have proved the formula for 16714 of the 16725 such curves of conductor less than 5000.
\end{abstract}

\maketitle

\section{Introduction}
Let $E$ be an elliptic curve defined over $\Q$, given by a global minimal Weierstrass equation. We denote the identity of $E$ by $\O$, the rank of the Mordell-Weil group $E(\Q)$ by $r$ and the conductor of $E$ by $N$. For each prime $p$, let $c_p(E)$ be the Tamagawa number at $p$ and let $a_p = p + 1 - \#\tilde E(\F_p)$ where $\tilde E(\F_p)$ is the mod-$p$ reduction of $E$. Let $L(E/\Q,s)$ be the Hasse-Weil $L$-function of $E$, and denote its order of vanishing at $s=1$ by $\ran(E/\Q)$. The regulator of $E(\Q)$ is denoted $\Reg(E(\Q))$. With $\omega$ denoting the minimal invariant differential let $\Omega(E) = \int_{E(\R)}|\omega|$ be the real period (the least positive real element of the canonical period lattice $\Lambda$) times the order of the component group of $E(\R)$ and let $||\omega||^2 = \int_{E(\C)}\omega\wedge\overline{i\omega}$ be twice the area of the fundamental domain of $\Lambda$. Denote the Shafarevich-Tate group by $\Sha(\Q, E)$ and for $G$ a group let $G_\tors$ denote its torsion subgroup and let $G_{/\tors}$ denote the quotient group $G/G_\tors$.

The Birch and Swinnerton-Dyer conjecture states that:
\begin{enumerate}
 \item The rank $r$ is equal to the analytic rank $\ran(E/\Q)$.
 \item The Shafarevich-Tate group is finite.
 \item The leading coefficient of the Taylor series of $L(E/\Q,s)$ at $s = 1$ is given by the formula:
\[
\frac{L^{(r)}(E/\Q,1)}{r!} = \frac{\Omega(E) \cdot \prod_p c_p(E) \cdot \Reg(E(\Q)) \cdot \#\Sha(\Q,E)}{\#E(\Q)_\tors^2}\,.
\]
\end{enumerate}

The first part of the conjecture, known as the rank conjecture, is one of the Clay Mathematics Institute's Millenium Prize Problems \cite{BSD_Clay}. It is known that if $\ran(E/\Q) \leq 1$ then the rank conjecture holds and the Shafarevich-Tate group is finite. It is worthy to note that $\#\Sha(\Q,E)_\mathrm{an}$, the value of $\#\Sha(\Q,E)$ for which the conjectural formula holds, is not even known to be a rational number for a single curve such that $\ran(E/\Q) > 1$. In this note we describe an algorithm which computes the order of the Shafarevich-Tate group of any elliptic curve $E$ such that $\ran(E/\Q) \leq 1$. This either proves the full conjecture for $E$ or produces a counterexample. We also report the results of computer calculations which prove the conjecture for 16714 of the 16725 such curves of conductor less than 5000.

\begin{defn} We denote by $\BSD(E/\Q,p)$ the following assertions:
\begin{enumerate}
 \item The rank $r$ is equal to the analytic rank $\ran(E/\Q)$.
 \item The $p$-primary part $\Sha(\Q,E)(p)$ of the Shafarevich-Tate group is finite.
 \item The real number $\#\Sha(\Q,E)_\mathrm{an}$ is rational.
 \item The conjectural formula holds at $p$, i.e.,
\[
\ord_p(\#\Sha(\Q,E)_\mathrm{an}) = \ord_p(\#\Sha(\Q,E)(p))\,.
\]
\end{enumerate}
We also denote $\BSD(E, p) = \BSD(E/\Q, p)$, and note that there is a definition of $\BSD(E/K, p)$ for global fields $K$ in general-- see, e.g., \cite{my-phd} for a definition.
\end{defn}
If $\ran(E/\Q) \leq 1$ then all but the last part of $\BSD(E/\Q,p)$ is known, so in this case $\BSD(E/\Q, p)$ is equivalent to the last equality. Clearly the Birch and Swinnerton-Dyer conjecture holds if and only if for each prime $p$, $\BSD(E/\Q,p)$ is true. The rank conjecture has been verified for $E/\Q$ of conductor $N < 130,000$ \cite{crem-130K}.

By the modularity theorem \cite{modularity1, modularity2}, every elliptic curve $E$ defined over $\Q$ has a modular parametrization $\psi : X_0(N) \rightarrow E$. If for each isogenous curve $E^\prime$ with modular parametrization $\psi^\prime : X_0(N) \rightarrow E^\prime$ we have that $\psi^\prime = \varphi \circ \psi$ for some isogeny $\varphi$ then we say that $E$ is an optimal elliptic curve, often called a strong Weil curve in the literature. Every elliptic curve over $\Q$ has an optimal elliptic curve in its isogeny class and by the characterizing property this curve is unique. Thus we can use optimal curves as isogeny class representatives and, by isogeny invariance of $\BSD(E,p)$ \cite{cassels-65}, focus on optimal curves.

\begin{thm}\label{thm_one}
Suppose $E/\Q$ is an elliptic curve of (analytic) rank at most 1 and conductor $N < 5000$. If $p$ is a prime such that $E[p]$ is irreducible then $\BSD(E, p)$ holds. If $E[p]$ is reducible and the pair $(E,p)$ is not one of the 11 pairs appearing in Table \ref{isogenies_left}, then $\BSD(E,p)$ holds.
\end{thm}
Note that this gives the full Birch and Swinnerton-Dyer conjecture for 16714 curves of the 16725 of rank at most one and conductor at most 5000. The remaining cases will be treated by a forthcoming paper by Michael Stoll and the author-- see Section \ref{reducible} for more details. This note was inspired by \cite{bsdalg}.

Whenever we prove a theorem with the help of a computer questions regarding errors both in hardware and software arise. Any computer-assisted proof implicitly includes as a hypothesis the statement that the software used did not encounter any bugs (hardware or software errors) during execution. Few software programs for serious number theory research have been proven correct. However it is often noted in the literature, as it is in Birch and Swinnerton-Dyer's seminal note \cite{bsd_notes_i} itself, that the kind of algorithms which occur in number theory (and more importantly the errors computational number theorists are likely to make implementing them) are often of a very particular sort. Either the software will work correctly or very quickly fail in an obvious way---perhaps it will crash or give answers that make no sense at all. In fact the computational work behind the theorems of Section \ref{proofs} uncovered several bugs (which have all been fixed). There are sometimes different implementations of the same algorithm or even different algorithms which implement the same theory. For example, the author used four different implementations of 2-descent to verify the computational claims of Theorem \ref{p_2}.

Throughout, $E$ will denote an elliptic curve defined over $\Q$, and we will be mainly interested in curves of rank 0 and 1. For such a curve, the Birch and Swinnerton-Dyer conjectural formula is known to hold up to a rational number, and sections 2 through 4 explain this result in such a way as to make it explicit. Sections 5 and 6 discuss what to do with the remaining primes and Section 7 contains the proof of Theorem \ref{thm_one}. Section 8 discusses the remaining cases, which all have reducible mod-$p$ representations.

The author wishes to thank John Cremona, Tom Fisher, Ralph Greenberg, Dimitar Jetchev, William Stein, Michael Stoll and Christian Wuthrich for their helpful comments and encouragement.

\section{Quadratic twists}\label{qit}
Below we will need to use several properties of the quadratic twist $E^d$ of the elliptic curve $E$ by a squarefree integer $d \not\in \{0, 1\}$, so we establish these here. For any number field $F$ let $G_F$ denote its absolute Galois group. Suppose $E$ is an elliptic curve over $\Q$ given in standardized ($a_1,a_3 \in \{0,1\}$ and $a_2 \in \{-1,0,1\}$) global minimal Weierstrass form
$$
E : y^2 + a_1 x y + a_3 y = x^3 + a_2x^2 + a_4 x + a_6\,.
$$
The curve $E^d$ can then be presented in the following Weierstrass form, which is not necessarily minimal:
\begin{align*}
E^d : y^2 &+ a_1xy + a_3y = x^3 \\&+ \left(a_2d + a_1^2(d-1)/4\right)x^2 \\&+ \left(a_4d^2 +a_1a_3(d^2-1)/2\right)x \\&+ a_6 d^3 + a_3^2(d^3 - 1)/4\,.
\end{align*}
Put $K = \Q(\theta)$ where $\theta^2 = 1/d$ and note that the curves are related by the $K$-isomorphism:
$$
\varphi : E \rightarrow E^d : \varphi(x,y) = \left(\theta^{-2}x, \theta^{-3}\left(y - \frac{a_1(\theta-1)}{2}x - \frac{a_3(\theta^3-1)}{2}\right) \right)\,.
$$
The $L$-series of $E/K$, $E/\Q$ and $E^d/\Q$ are related by the formula:
$$
L(E/K,s) = L(E/\Q,s)\cdot L(E^d/\Q,s)\,.
$$
Define as usual
\begin{align*}
 b_2 &= a_1^2 + 4a_2\,,  & c_4 &= b_2^2 - 24b_4\,,\\
 b_4 &= 2a_4 + a_1a_3\,, & c_6 &= b_2^3 + 36b_2b_4 - 216b_6\,,\\
 b_6 &= a_3^3 + 4a_6\,,  & \Delta &= (c_4^3 - c_6^2)/1728\,,
\end{align*}
noting that $\Delta$ is the minimal discriminant of $E$, hence $\omega = \text{d}x/(2y + a_1 x + a_3)$ is the minimal invariant differential of $E$. Let $\Delta'$ be the minimal discriminant of $E^d$, let $\text{sig}(E) = (\ord_2(c_4),\ord_2(c_6),\ord_2(\Delta))$ and for each prime $p$ let
$$\lambda_p = \min\{3\ord_p(c_4),2\ord_p(c_6),\ord_p(\Delta)\}\,.$$
The following proposition is a correction of \cite[Prop. 5.7.3]{connell_ech}:
\begin{prop}
 For each prime $p\mid 2d$ define $\delta_p$ as follows:
\begin{enumerate}
 \item If $p$ is odd, then define $\delta_p = 1$ if either $\lambda_p < 6$ or if $p = 3$ and $\ord_p(c_6) = 5$. Otherwise define $\delta_p = -1$.
 \item If $d \equiv 1 \pmod 4$ then $\delta_2 = 0$.
 \item If $d \equiv 3 \pmod 4$ then
 \begin{itemize}
  \item $\delta_2 = 2$ if $\text{sig}(E) = (0,0,\cdot)$ or $(\cdot,3,0)$,
  \item $\delta_2 = -2$ if $\text{sig}(E) = (4,6,c)$ with $c\geq 12$ and $2^{-6}c_6d \equiv -1 \pmod 4$, or if $\text{sig}(E) = (a,9,12)$ with $a \geq 8$ and $2^{-9}c_6d \equiv 1 \pmod 4$ and
  \item $\delta_2 = 0$ otherwise.
 \end{itemize}
 \item If $d \equiv 2 \pmod 4$ then
 \begin{itemize}
  \item $\delta_2 = 3$ if $\text{sig}(E) = (0,0,\cdot)$,
  \item $\delta_2 = -3$ if $\text{sig}(E) = (6,9,c)$ with $c \geq 18$ and $2^{-10}c_6d\equiv -1 \pmod 4$,
  \item $\delta_2 = 1$ if $\ord_2(c_4) \in \{4,5\}$, or if $\ord_2(c_6) \in \{3,5,7\}$, or if $\text{sig}(E) = (a,6,6)$ with $a \geq 6$ and $2^{-7}c_6d \equiv -1 \pmod 4$ and
  \item $\delta_2 = -1$ otherwise.
 \end{itemize}
\end{enumerate}
Then
$$\Delta' = \Delta\delta^6\qquad\text{where}\qquad\delta = \delta(E,d) = \prod_{p\mid 2d}p^{\delta_p}\,.$$
\end{prop}
The invariant differential $\omega_d$ associated to the given Weierstrass equation for $E^d$ has pullback $\varphi^*\omega_d = \theta\omega$ by \cite[p. 49]{Silverman_AEC}, and it may not be minimal. In fact, if $\Delta_d$ is the discriminant of the above equation for $E^d$ then $\theta^{12}\Delta_d = \Delta$. Since $\Delta = \Delta'\delta^{-6}$ we have $(\delta/d)^6\Delta_d = \Delta'$. The transformation taking $E^d$ to its minimal model must be defined over $\Q$ so $|\delta/d|\in\Q$ must be a square (or one can just read this off from the above proposition) and if $\omega'$ is the minimal invariant differential of $E^d$ then $\pm|\delta/d|^{-1/2}\omega_d = \omega'$. Finally since $\varphi^*\omega' = \pm|\delta/d|^{-1/2}\theta\omega$ we find the relationship between the canonical period lattices of $E$ and $E^d$:
$$
\Lambda_d = |\delta/d|^{-1/2}\theta\Lambda\,.
$$

Let $\sigma$ denote the nontrivial element of $G = \Gal(K/\Q)$. Define an action of $G$ on $H^1(K,E)$ by setting $\xi^\sigma(\tau) = \xi(\sigma\tau\sigma^{-1})^{\sigma}$ for $\tau \in G_K$ and $\{\xi\}^\sigma = \{\xi^\sigma\}$. Let $E(K)^\pm, H^1(K,E)^\pm$ denote the $\pm1$-eigenspaces of $E(K), H^1(K,E)$, respectively. By the definition of $E^d$ (see \cite[X \S 2]{Silverman_AEC}) we have that $\varphi^\sigma = [-1]\circ\varphi$. Then for $P \in E(K)$ we have
$$
P^\sigma = \pm P \quad \Rightarrow \quad \varphi(P)^\sigma = \varphi^\sigma(P^\sigma) = \mp\varphi(P)\,,
$$
and for $\xi : G_K \rightarrow E$ representing a cocyle class $\{\xi\} \in H^1(K,E)$ we have
$$
\{\xi^\sigma \pm \xi\} = 0 \quad \Rightarrow \quad \{(\varphi\circ\xi)^\sigma \mp \varphi\circ\xi\} = \{[-1]\circ\varphi\circ(\xi^\sigma\pm\xi)\} = 0\,,
$$
which show that $\varphi$ exchanges $E(K)^+$ with $E(K)^-$ and $H^1(K,E)^+$ with $H^1(K,E)^-$.

The following lemma gives a relationship between the Mordell-Weil groups $E(\Q), E^d(\Q)$ and $E(K)$.
\begin{lem}\label{qit_mw}
We have $E(\Q) = E(K)^+$ and under $\varphi^{-1}$ we may identify $E^d(\Q) = E(K)^-$. Under this identification we have:
\begin{enumerate}
\item The intersection is two torsion:
$$E(\Q) \cap E^d(\Q) = E(\Q)[2]\,,$$
\item if $E(K)$ has rank $r$ and $E(K)[2]$ has rank $s$, then
$$[E(K)_{/\tors} : (E(\Q) + E^d(\Q))_{/\tors}] \leq 2^r$$
and
$$[E(K) : E(\Q) + E^d(\Q)] \leq 2^{r+s}\,;$$
\item if $E(K)$ has rank 1, $E(\Q)$ has rank 0 and $E(\Q)[2] = 0$, then
$$E(K)_{/\tors} = E^d(\Q)_{/\tors}.$$
\end{enumerate}
\end{lem}
\begin{proof}The identifications are by definition and the above observations.
\begin{enumerate}
 \item Note that $P \in E(K)^+ \cap E(K)^-$ is equivalent to $P = P^\sigma = -P$.
 \item Let $P \in E(K)$ and note that 
$$
2P = (P + P^\sigma) + (P - P^\sigma) \in E(K)^+ + E(K)^-.
$$
Therefore since $2E(K) \subseteq E(K)^+ + E(K)^-$ we have that
$$
[E(K)_{/\tors} : (E(K)^+ + E(K)^-)_{/\tors}] \leq [E(K)_{/\tors} : 2E(K)_{/\tors}] = 2^r
$$
and
$$
[E(K) : E(K)^+ + E(K)^-] \leq [E(K) : 2E(K)] = 2^{r+s}.
$$
 \item Choose $P$ such that $E(K) = \Z P \oplus E(K)_\tors$. We have that $T := P^\sigma + P \in E(K)^+$ must be torsion, so choose $a, b$ so that the order of $T$ is $2^b(2a+1)$. With $W = P + aT$ we have that
$$
W^\sigma + W = P^\sigma + aT + (P + aT) = P^\sigma + P + 2aT = (2a+1)T
$$
must be in $E(\Q)(2)$ which is trivial since $E(\Q)[2] = 0$. Thus $W \in E(K)^-$ and since $W \equiv P$ modulo torsion we have $E(K)_{/\tors} = E(K)^-_{/\tors}$.
\end{enumerate}
\end{proof}

Recall the (exact) inflation-restriction sequence:
$$
0 \rightarrow H^1(K/\Q,E(K)) \stackrel{\text{inf}}{\rightarrow} H^1(\Q,E) \stackrel{\text{res}}{\rightarrow} H^1(K,E)\,.
$$
\begin{lem}\label{qit_sha}
 Up to a finite 2-group we may identify $H^1(\Q,E) = H^1(K,E)^+$ and $H^1(\Q,E^d) = H^1(K,E)^-$.
 \begin{enumerate}
  \item $H^1(K/\Q,E(K))$ is a finite 2-group.
  \item The image of $H^1(\Q,E)$ lies within $H^1(K,E)^+$.
  \item The quotient $H^1(K,E)/(H^1(K,E)^++H^1(K,E)^-)$ is a 2-group.
 \end{enumerate}
 Thus if $\Sha(K,E)$ is finite then up to a factor of 2,
$$\#\Sha(K,E) = \#\Sha(\Q,E) \cdot \#\Sha(\Q,E^d)\,.$$
\end{lem}
\begin{proof}
 We prove the first identification by establishing claims (1) and (2). The second comes from the analogous inf-res sequence for $E^d$ together with the fact that $\varphi$ takes $H^1(K,E)^-$ to $H^1(K,E^d)^+$.
 \begin{enumerate}
  \item A cocycle $\xi : G \rightarrow E(K)$ is determined by the image of $\sigma$ and since $0 = \xi(1) = \xi(\sigma) + \xi(\sigma)^\sigma$ we have $\xi(\sigma) \in E(K)^-$. Since $2\xi(\sigma) = (\xi(\sigma) + \xi(\sigma)^\sigma) + (\xi(\sigma) - \xi(\sigma)^\sigma) = \xi(\sigma) - \xi(\sigma)^\sigma$ every cocycle class is of order 2. Thus $H^1(K/\Q,E(K))$ is a 2-group which is a quotient of $E(K)^-$ and hence a finite 2-group.
  \item Suppose $\psi$ is a cocycle class representative in the image of the restriction map. Then there is a $\{\xi\} \in H^1(\Q,E)$ such that $\xi|_{G_K} = \psi$. Then one calculates for $\tau \in G_K$:
  \begin{align*}
   \psi^\sigma(\tau) - \psi(\tau) 
                                  &= \xi(\sigma\tau\sigma)^\sigma - \xi(\tau) \\
                                  &= \left( \xi(\sigma) + \xi(\tau\sigma)^\sigma \right)^\sigma - \xi(\tau) \\
                                  &= \xi(\sigma)^\sigma + \xi(\sigma)^\tau \\
                                  &= \xi(\sigma)^\tau - \xi(\sigma)\,.
  \end{align*}
  \item For $\{\xi\} \in H^1(K,E)$ we have $2\xi = (\xi + \xi^\sigma) + (\xi - \xi^\sigma)$.
 \end{enumerate}
\end{proof}

The following lemma is a generalization of a formula which appeared in \cite[p. 312]{gz86} without proof.
\begin{lem}\label{omega_relation}
Suppose $d < 0$ is a square-free integer. Then with $\delta = \delta(E, d)$ as defined above, we have:
$$
\Omega(E) \cdot \Omega(E^d) \cdot \delta^{1/2} = [E(\R) : E^0(\R)]\cdot ||\omega||^2.
$$
\end{lem}
\begin{proof}
Let $x$ be the least positive real element of the period lattice $\Lambda$, and choose a fundamental domain for $\Lambda$ with base $[0,x] \subset \R$ and upper left corner with positive imaginary part $y$ and real part in $[0,x)$. Then $\Omega(E)/[E(\R):E^0(\R)] = x$ and $||\omega||^2 = 2xy$. We compute
$$
\delta^{1/2}\Omega(E^d) = \delta^{1/2}\int_{E^d(\R)}|\omega'| = \int_{E(\C)^-}\delta^{1/2}|\varphi^*\omega'| = \int_{E(\C)^-}|\omega| = 2y\,.
$$
The claim follows.
\end{proof}

\section{Complex Multiplication}
Let $E$ be an elliptic curve over $\Q$, let $R$ denote the endomorphism ring $\End(E/\C)$, let $K$ denote its field of fractions and let $\Aut_R(E[p])$ denote the set of automorphisms of $E[p]$ commuting with the action of $R$. Consider the map $\overline \rho_{E,p} : G_\Q \rightarrow \Aut(E[p])$, which we call the mod-$p$ Galois representation.

If $E$ does not have complex multiplication, then $R = \Z$, $K = \Q$ and the groups $\Aut_R(E[p])$ and $\Aut(E[p]) \cong \GL_2(\F_p)$ are identical. If $E$ does have complex multiplication, then $R$ is an order in the quadratic imaginary field $K$ and we have $\overline \rho_{E,p}|_{G_K} : G_K \rightarrow \Aut_R(E[p]) \subsetneq \Aut(E[p])$. In either case we will say that $\overline \rho_{E,p}$ is \em surjective \em if the image of $G_K$ is $\Aut_R(E[p])$. Often in the literature one sees this defined as being ``as surjective as possible'' in the complex multiplication case. In Section \ref{examples} we will give several examples of this.

Note that there is always an isogeny defined over $\Q$ from $E$ to an elliptic curve $E'$ with complex multiplication by a maximal order. $E$ has complex multiplication by a non-maximal order if and only if its $j$-invariant is in the set $\{-12288000, 54000, 287496, 16581375\}$ \cite[p. 483]{silverman-advanced}.

We have the following theorem of Rubin:
\begin{thm}\label{CM_rubin}
Suppose $E$ is an elliptic curve defined over $K$ with complex multiplication by $\O_K$. With $w = \#\O_K^\times$ and $\tau \in \C^\times$ a generator of $\Lambda$, i.e., $\tau\O_K = \Lambda$, we have
\begin{enumerate}
\item If $L(E/K, 1) \neq 0$ then $E(K)$ is finite, $\Sha(K, E)$ is finite and there is a $u \in \O_K[w^{-1}]^\times$ such that
$$
L(E/K, 1) = \frac{\#\Sha(K, E)\cdot \tau\overline\tau}{u \cdot (\#E(K))^2}\,.
$$
\item If $L(E/K, 1) = 0$, then either $E(K)$ is infinite, or the $\p$-part of $\Sha(K, E)$ is infinite for all primes $\p \nmid \#\O_K^\times$.
\item If $E$ is defined over $\Q$ and $\ran(E/\Q) = 1$, then $\BSD(E/\Q, p)$ is true for all odd $p$ which split in $K$.
\end{enumerate}
\end{thm}
\begin{proof} See \cite{rubin-cm}. \end{proof}

\begin{cor}\label{CM_zero}
If $E$ is defined over $\Q$, has complex multiplication by $K$ and $\ran(E/\Q) = 0$, then $\BSD(E/\Q,p)$ is true for all $p \geq 5$. If $K \neq \Q(\sqrt{-3})$ then $\BSD(E/\Q,3)$ is true if and only if $3 \nmid \prod_pc_p(E)$.
\end{cor}
\begin{proof}
Let $K = \Q(\sqrt{d})$ for $d < 0$ a squarefree integer and without loss suppose $E$ has complex multiplication by $\O_K$. Letting $E^d$ be the model defined in Section \ref{qit} so that $\Lambda_d = \theta\Lambda$, note that since $[\theta^{-1}]$ is an endomorphism of $E$ we have $\theta^{-1}\Lambda \subset \Lambda$ which implies that $\Lambda \subset \Lambda_d$. Thus $E$ is isogenous to $E^d$, and this degree $|d|$ isogeny is defined over $\Q$ since its kernel is $\{z\Lambda : z \in \theta\Lambda\}$ and $G_\Q$ acts by multiplying $\theta$ by $\pm 1$. Thus we have that
$$
L(E/K,s) = L(E/\Q,s)^2\,.
$$

Since $L(E/\Q,1) \neq 0$ we have $L(E/K,1) \neq 0$ so $E(K)$ and $\Sha(K,E)$ are finite. By Lemmas \ref{qit_mw} and \ref{qit_sha} we have that $E(\Q)$ and $\Sha(\Q,E)$ are finite and
$$
L(E/\Q,1)^2 = \frac{\#\Sha(\Q,E)\cdot\#\Sha(\Q,E^d) \cdot \tau\overline\tau}{(\#E(\Q)\cdot\#E^d(\Q))^2}\cdot 2^wu^{-1}\,,
$$
where $w \in \Z$. In this situation we have \cite[Cor. 1.3]{cassels-65}:
$$
\frac{\#\Sha(\Q,E^d)}{\#E^d(\Q)^2} = \frac{\#\Sha(\Q,E)}{\#E(\Q)^2}\cdot\frac{\Omega(E)}{\Omega(E^d)}\prod_p\frac{c_p(E)}{c_p(E^d)}\,.
$$
By Lemma \ref{omega_relation} we have:
$$
\frac{L(E/\Q,1)^2}{\Omega(E)^2} = \frac{\#\Sha(\Q,E)^2}{\#E(\Q)^4}\cdot \frac{\tau\overline\tau\delta^{1/2}}{||\omega||^2}\cdot\frac{2^w}{[E(\R):E^0(\R)]\cdot u}\cdot\prod_p\frac{c_p(E)}{c_p(E^d)}\,.
$$
Note that $||\omega||^2/2$ is the area of $\Lambda = \tau\O_K$, i.e., $\tau\overline\tau$ times the area of $\O_K$. If $d \in \{-1,-2\}$ then this is $|d|^{1/2}$ and otherwise it is $|d|^{1/2}/4$. Therefore:
$$
\frac{\tau\overline\tau\delta^{1/2}}{||\omega||^2} = 2^v\delta^{1/2}|d|^{-1/2}\,.
$$
If $d \not\equiv 1 \pmod 4$ then $v = -1$ and $\delta/|d| \in \{4,1,1/4,1/16\}$. If $d \equiv 1 \pmod 4$ then $v = 1$ and $\delta \in \{|d|,1/|d|\}$. We will show that in this case $\delta = |d|$.

As noted in \cite[p. 176]{silverman-advanced}, since $E$ has complex multiplication it must be of additive reduction at all the bad primes. By \cite[Cor. 15.2.1, p. 359]{Silverman_AEC} the product $\prod_p c_p(E)$ is at most 4, and similarly for $\prod_pc_p(E^d)$. Since $c_p(E)/c_p(E^d) \in \{|d|,1,1/|d|\}$ for each prime $p$, if $d < -3$ then $\prod_pc_p(E)/c_p(E^d) = 1$ and if $d > -3$ then this product is a power of two. Now suppose $d < -3$ (hence $|d|$ is a prime) and note that since $L(E/\Q,1)/\Omega(E)$ is a rational number, we must have that $\ord_{|d|}((\delta/|d|)^{1/2})$ is even and this forces $\delta = |d|$ (this also implies that the other factors combine to make a perfect square, and that the ``error term'' $u$ is a rational number).

In summary, we have that
$$
L(E/\Q,1)^2 = \left(\frac{\#\Sha(\Q,E)\cdot \Omega(E)}{\#E(\Q)^2}\right)^2\cdot \left(\frac{2^{v+w}\delta^{1/2}}{[E(\R):E^0(\R)]\cdot|d|^{1/2}}\right) \cdot
\left(\frac1u\prod_p\frac{c_p(E)}{c_p(E^d)}\right)\,,
$$
where the second factor is in $2^{2\Z}$ and the third factor is in $2^{2\Z}$ if $d \neq -3$ and in $2^{2\Z}3^{2\Z}$ if $d = -3$. Since $\Reg(E(\Q)) = 1$ the claim follows.
\end{proof}

\begin{prop}
 If $E/\Q$ has rank 0, has complex multiplication by $K \neq \Q(\sqrt{-3})$ and has conductor $N < 130000$, then $\BSD(E/\Q, 3)$ is true.
\end{prop}
\begin{proof}
 In this situation it is easy to check that $3 \nmid \prod_pc_p$ using computation.
\end{proof}

\begin{lem}\label{rubin-cyclic}
With $E$ and $K$ as above, let $\p$ be a prime of $K$ of good reduction for $E$ which does not divide $\#\O_K^\times$. Then $K(E[\p])/K$ is a cyclic extension of degree $\mathrm{Norm}(\p) - 1$ in which $\p$ is totally ramified.
\end{lem}
\begin{proof} See \cite[Lemma 21(i)]{rubin-cm2}. \end{proof}

\begin{lem}\label{aut-o}
With $E$ and $K$ as above, we have $(\O_K/p\O_K)^\times \cong \Aut_{\O_K}(E[p])$ for all primes $p$.
\end{lem}
\begin{proof}
Let $\O_K = \Z[\alpha]$. Then via the isomorphism $E[p] \cong \F_p^2$, the element $\alpha$ acts on $\F_p^2$ by a matrix $M \in \GL_2(\F_p)$ and $\Aut_{\O_K}(E[p])$ is isomorphic to the centralizer of $M$ in $\GL_2(\F_p)$. The centralizer of $M$ is equal to $\F_p[M]$ since $M$ cannot be a scalar element. In other words, we can make the identification $\Aut_{\O_K}(E[p]) = \F_p[\alpha]$ by viewing $\alpha$ as an element of $\Aut(E[p])$. We define an isomorphism $\Aut_{\O_K}(E[p]) \rightarrow (\O_K/p\O_K)^\times$ by sending $\alpha^n$ to $\alpha^n + p\O_K \in (\O_K/p\O_K)^\times$. If $\alpha^n \in p\O_K$ then $M^n = 0$ in $\GL_2(\F_p)$, hence the map is injective. It is surjective since $\O_K = \Z[\alpha]$.
\end{proof}

\begin{prop}\label{CM_surjective}
If $p$ is a prime of good reduction for $E$ not dividing $\#\O_K^\times$ which is inert in $K$, then $\overline \rho_{E,p}$ is surjective.
\end{prop}
\begin{proof}
When $p$ is inert in $K$, $\mathrm{Norm}(\p) = p^2$. By Lemma \ref{rubin-cyclic} $\#\Gal(K(E[p])/K) = p^2 - 1$. Since $\overline \rho_{E,p} : \Gal(K(E[p])/K) \rightarrow \Aut_{\O_K}(E[p])$ is injective it suffices to show that $\#\Aut_{\O_K}(E[p]) = p^2 - 1$. By Lemma \ref{aut-o} this reduces to showing $\#(\O_K/p\O_K)^\times = p^2 - 1$ which is true since $[\O_K/p\O_K : \Z/p\Z] = 2$.
\end{proof}

The following result will be useful:
\begin{thm} \label{CMtorsion}
 Let $E$ be an elliptic curve defined over $\Q$ with complex multiplication by an order of $K = \Q(\sqrt{-d})$ and $p$ an odd prime not dividing $d$. Let $F$ be a Galois number field not containing $K$. Then $E(F)[p]$ is trivial.
\end{thm}
\begin{proof}
 This is \cite[Theorem 2]{CMtorsion}.
\end{proof}

\section{Heegner Points} \label{heeg_pts}
If $E$ is an elliptic curve over $\Q$ of conductor $N$, we say that the quadratic imaginary field $K = \Q(\sqrt{D})$ satisfies the Heegner hypothesis for $E$ if each prime $p \mid N$ splits in $K$. If $K$ satisfies the Heegner hypothesis for $E$, then the Heegner point $y_K \in E(K)$ is defined as follows (see \cite{gross-kolyvaginswork} for details). By hypothesis there is an ideal $\mathcal N$ of $\O_K$ such that $\O_K/\mathcal N$ is cyclic of order $N$. Since $\O_K \subset \mathcal N^{-1}$, we have a cyclic $N$-isogeny $\C/\O_K \rightarrow \C/\mathcal N^{-1}$ of elliptic curves with complex multiplication by $\O_K$ and hence a point $x_1 \in X_0(N)$. By the theory of complex multiplication $x_1$ is defined over the Hilbert class field $H$ of $K$. We fix a modular parametrization $\psi : X_0(N) \rightarrow E$ of minimal degree taking $\infty$ to $\O$, which exists by \cite{modularity1} and \cite{modularity2}. As above denote the minimal invariant differential on $E$ by $\omega$. Then $\psi^*(\omega)$ is the differential associated to a newform on $X_0(N)$. We have $\psi^*(\omega) = \alpha \cdot f$ where $f$ is a normalized cusp form and $\alpha$ is some nonzero integer \cite{edix-91} constant. The Manin constant is $c := |\alpha|$ and the Heegner point is $y_K := \Tr_{H/K}(\varphi(x_1)) \in E(K)$. It has been conjectured that $c = 1$ if $E$ is optimal, and this has been verified for $N < 130000$ by Cremona in \cite{manin_const}. Define $I_K := [E(K)_{/\tors} : \Z y_K]$, which we call the Heegner index. Note that sometimes we may denote the Heegner index by $I_D$ to emphasize the dependence $K = \Q(\sqrt{D})$.

Gross, Zagier and Zhang have proven a deep theorem which expresses the first derivative of the $L$-series of $E/K$ at 1 in terms of the canonical height $\hat h$ of the Heegner point $y_K$.
\begin{thm}[Gross-Zagier-Zhang]\label{gzz}
If $K$ satisfies the Heegner hypothesis for $E$, then
$$
L'(E/K, 1) = \frac{2||\omega||^2 \hat h(y_K)}{c^2 \cdot u_K^2 \cdot \sqrt{|\Delta(K)|}},
$$
where $||\omega||^2 = \int_{E(\C)}\omega \wedge \overline{i\omega}$ and the quadratic imaginary number field $K$ has $2u_K$ roots of unity and discriminant $\Delta(K)$.
\end{thm}
\begin{proof} Gross and Zagier first proved this in \cite{gz86} when $D$ is odd and Zhang generalized it in \cite{zhang04}. \end{proof}

Note that $u_{\Q(\sqrt{-1})} = 2$, $u_{\Q(\sqrt{-3})} = 3$ and for all other quadratic imaginary fields $K$ we have $u_K = 1$. Often one requires that $D \not \in \{-1,-3\}$, but since there are infinitely many $D$ satisfying the Heegner hypothesis for $E$ if $\ran(E) \leq 1$, this is a minor issue (see the proof of Theorem \ref{big-result}). Note also that the $\hat h$ appearing in the formula as stated here is the absolute height, whereas the one appearing in \cite[Theorem 2.1, p. 311]{gz86} is equal to our $2\hat h$.

We have the following theorem of Kolyvagin:
\begin{thm}\label{kolyvagin_big}
If $y_K$ is nontorsion, then $E(K)$ has rank 1 (hence $I_K < \infty$), $\Sha(K, E)$ is finite and
$$
c_3I_K\Sha(K, E) = 0\text{ and }\#\Sha(K, E) \big| c_4 I_K^2,
$$
where $c_3$ and $c_4$ are positive integers (explicitly defined in \cite{koly90}). The primes dividing $c_4$ are at most 2 and the odd primes $p$ for which $\overline \rho_{E,p}$ is surjective.
\end{thm}
\begin{proof} This is \cite[Theorem A]{koly90}. \end{proof}

\begin{cor}
If $y_K$ is nontorsion, then $\Sha(\Q, E)$ and $\Sha(\Q, E^D)$ are finite and have orders whose odd parts divide $c_4I_K^2$.
\end{cor}
\begin{proof} By Lemma \ref{qit_sha} we have that $\#\Sha(\Q, E)\cdot\#\Sha(\Q, E^D)$ divides $\#\Sha(K, E)$ up to a power of two. \end{proof}

\begin{thm} \label{big-result}
If $\ran(E) \leq 1$, then $y_K$ is nontorsion. In particular,
$$
r(E) = \ran(E),
$$
$\Sha(\Q, E)$ is finite, and if $p$ is an odd prime unramified in the CM field such that $\overline \rho_{E,p}$ is surjective, then
$$
\ord_p(\#\Sha(\Q, E)) \leq 2 \cdot \ord_p(I_K).
$$
\end{thm}
\begin{proof}
We follow the proof given in \cite{darm-rat}. If $\varepsilon = -1$ (i.e., $\ran(E) = 1$), then a result of Waldspurger (see \cite{wald-85}) implies that there are infinitely many $D < 0$ such that $K = \Q(\sqrt{D})$ satisfies the Heegner hypothesis for $E$ and $\ran(E^D) = 0$. If $\varepsilon = 1$  (i.e., $\ran(E) = 0$), then results of Bump, Friedberg and Hoffstein (see \cite{bump-friedberg-hoffstein-90}) or independently results of Murty and Murty (see \cite{murty-murty}) imply that there are infinitely many $D < 0$ such that $K = \Q(\sqrt{D})$ satisfies the Heegner hypothesis for $E$. In this case, for parity reasons, $L(E^D/\Q, 1)$ is always 0.

We have that
$$
\ord_{s=1}L(E/K,s) = \ord_{s=1}L(E/\Q,s) + \ord_{s=1}L(E^D/\Q,s),
$$
which implies that in either case $\ran(E/K) = 1$ which, by the Gross-Zagier-Zhang formula (Theorem \ref{gzz}), implies that $y_K$ is nontorsion. Then Kolyvagin's theorem implies that $E(K)$ has rank 1, $I_K < \infty$ and that $\Sha(K, E)$ is finite.

By Lemma \ref{qit_mw}, we have $$\rank(E(K)) = \rank(E(\Q)) + \rank(E^D(\Q)).$$ The point $y_K$ belongs to $E(\Q)$ (up to torsion) if and only if $\varepsilon = -1.$ If $\varepsilon = -1$, then $\rank(E(\Q)) = 1$ since $y_K \in E(\Q)_{/\tors}$. If $\varepsilon = 1$, then some multiple of $y_K$ is in $E(K)^-$, which implies that $\rank(E^D(\Q)) = 1$, hence $\rank(E(\Q)) = 0$.
\end{proof}

\begin{thm}\label{big_CM_thm}
Suppose $E$ has CM by the full ring of integers $\O_K$.
\begin{enumerate}
\item If $\ran(E) = 0$, then $\BSD(E/\Q, p)$ is true for $p \geq 5$.
\item If $\ran(E) = 1$, then:
\begin{enumerate}
\item If $p \geq 3$ is split, then $\BSD(E/\Q,p)$ is true.
\item If $p \geq 5$ is inert and $p$ is a prime of good reduction for $E$, then
$$
\ord_p(\#\Sha(\Q,E)) \leq 2 \cdot \ord_p(I),
$$
where $I = I_{\Q(\sqrt{D})}$ is any Heegner index for $D < -4$ satisfying the Heegner hypothesis.
\end{enumerate}
\end{enumerate}
\end{thm}
\begin{proof}
Part (1) is Corollary \ref{CM_zero}. Part (2a) is part (3) of Theorem \ref{CM_rubin}. Part (2b) is obtained from Proposition \ref{CM_surjective} by Theorem \ref{big-result}.
\end{proof}

We now describe an algorithm for computing the Mordell-Weil and Shafarevich-Tate groups when the analytic rank of $E/\Q$ is bounded above by one. In the next section we will make this more explicit, with the aim of developing a practical procedure for verifying the Birch and Swinnerton-Dyer conjecture for a specific elliptic curve.
\begin{lem} \label{height-bound}
If $B > 0$ is such that $S = \{P \in E(\Q) : \hat h(P) \leq B\}$ contains a set of generators for $E(\Q)/2E(\Q)$ then $S$ generates $E(\Q)$.
\end{lem}
\begin{proof} See \cite[\S3.5]{cremona-algs}.
\end{proof}

\begin{thm}
If $\ran(E) \leq 1$, then there are algorithms to compute both the Mordell-Weil group $E(\Q)$ and the Shafarevich-Tate group $\Sha(\Q, E)$.
\end{thm}
\begin{proof}
In general 2-descent is not known to terminate, but in this case $r = \ran(E)$ is known. Therefore 2-descent will determine $E(\Q)/2E(\Q)$. Then we can search for points up to the maximum height of points in $E(\Q)/2E(\Q)$ and by Lemma \ref{height-bound} we will find a set of generators for $E(\Q)$.

To compute $\Sha(\Q, E)$, note that Kolyvagin's theorem gives an explicit upper bound $B$ for $\#\Sha(\Q, E)$. For primes $p$ dividing this upper bound, we can (in theory at least) perform successive $p^k$-descents for $k = 1, 2, 3, ...$ to compute $\Sha(\Q, E)[p^k]$. As soon as $\Sha(\Q, E)[p^k] = \Sha(\Q, E)[p^{k+1}]$ we have $\Sha(\Q, E)[p^k] = \Sha(\Q, E)[p^\infty]$ and can move on to the next prime. Once we do this for each prime we have $\Sha(\Q, E) = \bigoplus_{p \mid B}\Sha(\Q, E)[p^\infty]$.
\end{proof}

For $\ran(E) \leq 1$, we can (at least in theory) compute $\#\Sha(\Q, E)_\text{an}$ exactly, as first described in \cite[p. 312]{gz86}. Together with the previous theorem, this shows that the BSD formula for $E$ can be proven for specific elliptic curves via computation.

The main ingredient to applying Kolyvagin's work to a specific elliptic curve $E$ of analytic rank at most 1 is to compute the Heegner index $I_K = [E(K)_{/\tors}:\Z \overline{y_K}]$, where $K = \Q(\sqrt{D})$ satisfies the Heegner hypothesis for $E$ and $y_K \in E(K)$ is a Heegner point (and $\overline{y_K}$ is its image in $E(K)_{/\tors}$). Let $z \in E(K)$ generate $E(K)_{/\tors}$. 

We can efficiently compute $\hat h(y_K)$ to desired precision using the Gross-Zagier-Zhang formula (Theorem \ref{gzz})), reducing the index calculation to the computation of the height of $z$, since
$$
I_K^2 = \frac{\hat h(y_K)}{\hat h(z)}.
$$
We have the following corollary of Lemma \ref{qit_mw}:
\begin{cor}\label{bounding_heeg_ind}
Suppose $E$ is an elliptic curve of analytic rank 0 or 1 over $\Q$, in particular $\rank(E(\Q)) = \ran(E(\Q))$. Let $D < 0$ be a squarefree integer such that $K = \Q(\sqrt{D})$ satisfies the Heegner hypothesis for $E$.
\begin{enumerate}
\item If we have $\ran(F(\Q)) = 1$, where $F$ is one of $E$ or $E^D$, and if $x \in F(\Q)$ generates $F(\Q)_{/\tors}$, then
$$
I_K = \begin{cases}
\sqrt{\frac{\hat h(y_K)}{\hat h(x)}}, & \frac12x \not \in F(K),\\
2\sqrt{\frac{\hat h(y_K)}{\hat h(x)}}, & \frac12x \in F(K).
\end{cases}
$$
\item Suppose $\ran(E(\Q)) = 0$. If $E(\Q)[2] = 0$ then let $A = 1$, otherwise let $A = 4$. Let $C = C(E^D/\Q)$ denote the Cremona-Pricket-Siksek height bound \cite{height-bound}. If there are no nontorsion points $P$ on $E^D(\Q)$ with naive absolute height
$$
h(P) \leq \frac{A\cdot\hat h(y_K)}{M^2} + C,
$$
then
$$
I_K < M.
$$
\end{enumerate}
\end{cor}
Note that this is a correction to the results stated in \cite{bsdalg}. However, for each case in which \cite{bsdalg} uses this result, the corresponding $A$ is equal to 1. Therefore this mistake does not impact any of the other results there.

If $\rank(E(\Q)) = 1$, then we will have a generator $x$ from the rank verification, and we can simply check whether $\frac12x$ is in $E(K)$ and use part 1 of the corollary. If $\rank(E(\Q)) = 0$ then we may not so easily find a generator of the twist, because a point search may very well fail since the conductor of $E^D$ is $D^2N$. However, a failed point search can still be useful as long as we search sufficiently hard, because of part 2 of the corollary.

\section{Bounding the order of $\Sha(\Q,E)$}\label{comp-sha}

Suppose $\ran(E) \leq 1$ for $E/\Q$ and that $K$ is a quadratic imaginary field satisfying the Heegner hypothesis for $E$. We have already seen that for analytic rank zero curves $\BSD(E, p)$ is true for primes $p > 3$ if $E$ has complex multiplication. Otherwise we have the following theorem:
\begin{thm} \label{kato}
Suppose $E$ is an optimal non-CM curve, and let $p$ be a prime such that $p \nmid 6 N$ and $\rho_{E,p}$ is surjective. If $\ran(E) = 0$ then $\Sha(\Q,E)$ is finite and
$$
\ord_p(\#\Sha(\Q,E)) \leq \ord_p\left(\frac{L(E/\Q,1)}{\Omega(E)}\right).
$$
\end{thm}
\begin{proof} As outlined in \cite[\S 4]{bsdalg}, this is due to Kato's Euler system \cite{kato04} together with a result of Matsuno \cite{matsuno03}. \end{proof}
As a corollary to this theorem $\BSD(E, p)$ is true for primes $p > 3$ of good reduction where $E[p]$ is surjective and $p$ does not divide $\#\Sha(\Q,E)_\text{an}$. Under certain technnical conditions on $p$ (explained in \cite{grigorov_phd}), Grigorov has proven the bound on the other side:
$$
\ord_p(\#\Sha(\Q,E)) = \ord_p\left(\frac{L(E/\Q,1)}{\Omega(E)}\right).
$$

Because Theorem \ref{kato} often eliminates most of the primes $p > 3$, one often does not need to compute the Heegner index for rank zero curves. However, if there is a bad prime $p > 3$ such that $\overline\rho_{E,p}$ is surjective then Theorem \ref{kato} does not apply and descents are in general not feasible. For example, this happens with the pair $(E,p) = (\text{2900d1}, 5)$. Interestingly $\#\Sha(\Q,E) = 25$ in this case (this will be proven in Section \ref{proofs}). Theorem \ref{big-result} still gives an upper bound in this case, provided we have some kind of bound on the Heegner index. In the example above the methods of Section \ref{heeg_pts} show that $I_K \leq 23$, implying that $\ord_5(I_K) \leq 1$ and hence $\ord_5(\#\Sha(\Q,E))\leq 2$.

The following theorems give alternate hypotheses under which Kolyvagin's machinery still gives the same result. These should be viewed as extensions of Theorem \ref{big-result}.
\begin{thm} \label{cha}
If $\ran(E/\Q) \leq 1$ and $p$ is a prime such that $p \nmid 2\cdot \Delta(K)$, $p^2 \nmid N$ and $\overline \rho_{E,p}$ is irreducible then
$$
\ord_p(\#\Sha(\Q, E)) \leq 2 \cdot \ord_p(I_K).
$$
\end{thm}
\begin{proof} See \cite{cha03,cha05}. \end{proof}

\begin{thm}\label{stein-et-al}
Suppose $\ran(E/\Q) \leq 1$ and $E$ is non-CM and suppose $p$ is an odd prime which does not divide $\#E'(\Q)_\tors$ for any $E'$ which is $\Q$-isogenous to $E$. If $\Delta(K)$ is divisible by exactly one prime, further suppose that $p \nmid \Delta(K)$. Then
$$
\ord_p(\#\Sha(\Q, E)) \leq 2 \cdot \ord_p(I_K).
$$
\end{thm}
\begin{proof} See \cite[Thm. 3.5]{bsdalg}. \end{proof}

Jetchev \cite{jetchev_m_max} has improved the upper bound with the following:
\begin{thm}[Jetchev] \label{jetchev}
If the hypotheses of any of Theorems \ref{big-result}, \ref{cha} or \ref{stein-et-al} apply to $p$, then
$$
\ord_p(\#\Sha(\Q, E)) \leq 2 \cdot \left(\ord_p(I_K) - \max_{q\mid N} \ord_p(c_q)\right).
$$
If $p$ divides at most one Tamagawa number then this upper bound is equal to $\ord_p(\#\Sha(\Q,E)_\text{an})$.
\end{thm}


There is also an algorithm of Stein and Wuthrich based on the work of Kato, Perrin-Riou and Schneider (a preprint is available at \cite{shark} and the algorithm is implemented in Sage \cite{Sage}). Suppose that the elliptic curve $E$ and the prime $p \neq 2$ are such that $E$ does not have additive reduction at $p$ and the image of $\overline \rho_{E,p}$ is either equal to the full group $\GL_2(\F_p)$ or is contained in a Borel subgroup of $\GL_2(\F_p)$. (In $\GL_2(\F_p)$ these are subgroups which are conjugate to the group of upper triangular matrices. See \cite[Section 21]{humphreys} for more details.) These conditions hold for all but finitely many $p$ if $E$ does not have complex multiplication. Given a pair $(E, p)$ satisfying this hypothesis, the algorithm either gives an upper bound for $\#\Sha(\Q,E)[p^\infty]$ or terminates with an error. In the case that $\ran(E) \leq 1$, an error only happens when the $p$-adic height pairing can not be shown to be nondegenerate. For curves of conductor up to 5000 and of rank 0 or 1 this never happens for those $p$ considered. Note that it is a standard conjecture that the $p$-adic height pairing is nondegenerate, and if this is true for a particular case, it can be shown via a computation.

There are also techniques for bounding the order of $\Sha(\Q,E)$ from below. In \cite{crem-maz-vis}, Cremona and Mazur establish a method for visualizing pieces of $\Sha(\Q,E)$ as pieces of Mordell-Weil groups via modular congruences, which is fully explained in the appendix of \cite{agashe-stein-vis}. They have also carried out computations for curves of conductor up to 5500, which are listed in \cite{crem-maz-vis}. In addition, Stein established a method for doing this for abelian varieties as part of his Ph.D. thesis \cite{stein-phd}.

\section{Examples} \label{examples}
The following examples are not only useful in illustrating the preceding discussion, but will also be needed to prove the main results of this note. We begin by proving that several mod-$p$ Galois representations are surjective, where the elliptic curve has complex multiplication. This will allow us to use Theorem \ref{big-result} for these curve-prime pairs in Section \ref{proofs}.

\begin{example} \label{CM_rk1_5}
 Let $p = 5$. \begin{align*}
  \text{675a1}  & : y^2 + y = x^3 + 31\\
  \text{900c1}  & : y^2 = x^3 + 100\\
  \text{2700h1} & : y^2 = x^3 + 625\\
  \text{2700l1} & : y^2 = x^3 + 5\\
  \text{2700p1} & : y^2 = x^3 + 500\\
  \text{3600bd1}& : y^2 = x^3 - 100
 \end{align*}
 \begin{enumerate}
  \item Let $E$ be the curve 675a, which has complex multiplication by the full ring of integers of $K = \Q(\sqrt{-3})$ (which is generated over $\Z$ by $\alpha = (\sqrt{-3}+1)/2$). After a choice of basis for $E[p]$, we find that $\alpha$ acts on $E[p] \cong (\Z/p\Z)^2$ via the matrix $$M = \left(\begin{array}{rr} 4 & 3 \\ 4 & 2 \end{array}\right).$$
  The centralizer of $M$ in $\GL_2(\Z/p\Z)$ is of order $p^2 - 1$, whence $$\#\Aut_R(E[p]) = 24.$$

  Let $f(x)$ be the $p$-division polynomial of $E$, and let $g(y)$ be the resultant with respect to $x$ of $f(x)$ and the defining polynomial of $E$, noting that the roots of $f$ are the $x$-coordinates of the points of $E[p]$ and those of $g$ are the $y$-coordinates. Over $\Q$, $f(x)$ is irreducible of degree 12 and $g(y) = g_1(y)^3$, where $g_1(y)$ is irreducible of degree 8. If $K_x = \Q[x]/(f(x))$ and $K_y = \Q[y]/(g_1(y))$, then $f$ factors into four linear factors and four quadratic ones over the compositum $K_x \cdot K_y$ and $g_1$ into four linear factors and two quadratic ones. One can verify that $K \not\subset K_x \cdot K_y$ and so $K_x \cdot K_y \cdot K$, which has degree 48, is a subfield of $\Q(E[p])$ by Theorem \ref{CMtorsion}. On the one side since $\Gal(K(E[p])/K)$ is a subgroup of $\Aut_R(E[p])$ we have $[K(E[p]) : K] \mid 24$, while on the other side since $K_x \cdot K_y \cdot K \subseteq \Q(E[p]) \subseteq K(E[p])$ we have $24 \mid [K(E[p]) : K]$. Therefore $\#\Gal(K(E[p])/K) = 24$ and hence $\overline \rho_{E,p}$ is surjective.

  If $E$ is one of the curves 900c, 2700h, 2700l or 2700p, then we have the same $K$, the same factoring patterns, the same conjugacy class in $\GL_2(\F_p)$ and in each of these cases $K \not\subset K_x \cdot K_y$. Therefore for each of these, $\overline \rho_{E,p}$ is surjective.

  \item Let $E$ be the curve 3600bd, which also has the same $K$. The matrix is now $$M = \left(\begin{array}{rr} 4 & 3 \\ 4 & 2 \end{array}\right),$$ which still has centralizer of order 24. Computing the compositum $K_x \cdot K_y$ is difficult in this case, so we argue differently. The intersection of $K_x$ and $K_y$ is degree 4, and one can verify that $K$ is not contained in $K_y$, which is of degree 8. Therefore we can still conclude that $K$ is not contained in $K_x \cdot K_y$. Now we may proceed as above, and conclude that $\overline \rho_{E,p}$ is surjective.
 \end{enumerate}
\end{example}

\begin{example} \label{CM_rk1_7}
 Let $p = 7$. \begin{align*}
  \text{1568g1} & : y^2 = x^3 - 49x\\
  \text{3136t1} & : y^2 = x^3 + 49x\\
  \text{3136u1} & : y^2 = x^3 - 343x\\
  \text{3136v1} & : y^2 = x^3 - 7x
 \end{align*}
 If $E$ is one of the curves 1568g, 3136u or 3136v, then the matrix is $$M = \left(\begin{array}{rr} 0 & 1 \\ 6 & 0 \end{array}\right),$$ whereas if $E$ is the curve 3136t, then the matrix is $$M = \left(\begin{array}{rr} 6 & 2 \\ 6 & 1 \end{array}\right).$$ In all cases the centralizer is order 48, and $K = \Q(i)$. With $f(x)$, $g(y)$ and $K_x$ as in the previous example, $g(y)$ is irreducible of degree 48 in each case (so let $K_y = \Q[y]/(g(y))$), and it is possible to verify that $K \not\subset K_y$. Again, by Theorem \ref{CMtorsion}, we have that $K \cdot K_y \subseteq Q(E[p]) \subseteq K(E[p])$, and that the compositum is degree 96. Therefore $\#\Gal(K(E[p])/K) = 48$ and $\overline \rho_{E,p}$ is surjective.
\end{example}

\begin{example} \label{CM_rk1_11}
 Let $p = 11$. \begin{align*}
  \text{3267d1} & : y^2 + y = x^3 - 333\\
  \text{3872a1} & : y^2 = x^3 + 1331x\\
  \text{4356a1} & : y^2 = x^3 - 44\\
  \text{4356b1} & : y^2 = x^3 + 58564\\
  \text{4356c1} & : y^2 = x^3 - 1331
 \end{align*}
 \begin{enumerate}
  \item If $E$ is one of the curves 3267d, 4356a, 4356b or 4356c, then $K = \Q(\sqrt{-3})$ and the matrices are, respectively, $M_1, M_1, M_2$ and $M_1$, where $$M_1 = \left(\begin{array}{rr} 0 & 1 \\ 10 & 1 \end{array}\right),$$ and $$M_2 = \left(\begin{array}{cc} 10 & 3 \\ 10 & 2 \end{array}\right).$$ In all cases the centralizer is of order 120. With $f(x)$, $g(y)$ and $K_x$ as above, then over $\Q$, $f(x)$ is irreducible of degree 60 and $g(y) = g_1(y)^3$, where $g_1(y)$ is irredcubile of degree 40 (so let $K_y = \Q[y]/(g_1(y))$). We can verify that $f$ and $g_1$ remain irreducible over $K$, and so $[K_x \cdot K : \Q] = 120$ and $[K_y \cdot K : \Q] = 80$. Since the least common multiple is 240, we have $[K_x \cdot K_y \cdot K : \Q] = 240$. Therefore $\#\Gal(K(E[p])/K) = 120$ and $\overline \rho_{E,p}$ is surjective.
  \item If $E$ is 3872a, then $K = \Q(i)$ and the matrix is $$M = \left(\begin{array}{cc} 10 & 2 \\ 10 & 1 \end{array}\right).$$ The centralizer is again of order 120. With $f(x)$, $g(y)$ and $K_x$ as above, then over $\Q$, $f(x)$ is irreducible of degree 60 and $g(y)$ is irreducible of degree 120 (so let $K_y = \Q[y]/(g(y))$). Since $g(y)$ remains irreducible over $K$ we have $[K_y \cdot K : \Q] = 240$. Therefore $\#\Gal(K(E[p])/K) = 120$ and $\overline \rho_{E,p}$ is surjective.
 \end{enumerate}
\end{example}

Schaefer and Stoll have described a way of computing the $p$-Selmer group in \cite{schaefer-stoll}. If $S$ is $\{p\}$ union the set of primes $\ell$ or such that $p$ divides $c_\ell$, then $\Sel^{(p)}(\Q, E)$ corresponds to the subgroup of elements of $H^1(\Q, E; S)$ whose localizations are in the image of the local connecting homomorphisms for each place in $S$. In practice, one computes the $S$-Selmer group $K(S, p)$ of the \'etale algebra $K$ corresponding to a Galois-invariant subset of $E[p] \setminus \{\O\}$ in terms of the class group and $S$-units. Here we give two useful examples of this techinque, which proves that the 5-primary part of $\Sha(\Q,E)$ is trivial.

\begin{example} \label{CM_rk1_5descent} Let $p = 5$.
 Usually five-descents are infeasible due to the number fields involved, e.g. if the mod-5 representation is surjective, the \'etale algebra will be a single number field of degree 24, for which class group and $S$-unit calculations will be too difficult to complete without assuming GRH. However, the following two examples illustrate cases in which a five descent is actually possible without assuming GRH. Here the 5-division polynomial has a factor of degree 4 which corresponds to a Galois invariant spanning subset $X$ of $E[p] \setminus \{\O\}$ of size 8. In each case $g(y)$ is the resultant of this factor and the defining polynomial of $E$, which defines a number field $A_1$.

 \begin{enumerate}
  \item Let $E$ = 225a1. Then we have
  \begin{align*}
  g(y) = &\  y^8 + 4y^7 + 97y^6 + 277y^5 - 80y^4 \\ &\ \ \ \  - 617y^3 - 548y^2 - 194y + 331.
  \end{align*}

  \item Let $E$ = 3600be. Then we have
  \begin{align*}
  g(y) = &\  y^8 - 720000y^6 - 27000000000y^4 \\ &\ \ \ \  + 1458000000000000000.
  \end{align*}
  
 \end{enumerate}

 In both cases the set $S$ is of order one, consisting of the prime above 5, and the dimension of $A_1(S, p)$ is 6. Computations show that the dimension of $A_1(S, p)^{(1)} = \ker(\sigma_g - g)$ (notation again comes from \cite{schaefer-stoll}) is at most 2 in both cases. Since the Selmer group $\Sel^{(5)}(\Q, E)$ is contained in $A_1(S, p)^{(1)}$ it is dimension at most 2, and since the dimension of $E(\Q)/5E(\Q)$ is exactly 1, we have that in these two cases $\#\Sha(\Q, E)[5] \leq 5$, and hence that $\#\Sha(\Q, E)[5] = 1$.


\end{example}

\section{Curves of conductor $N < 5000$, irreducible mod-$p$ representations}\label{proofs}

There are 17314 isogeny classes of elliptic curves of conductor up to 5000. There are 7914 of rank 0, 8811 of rank 1, 589 of rank 2, and none of higher rank. There are only 116 optimal curves which have complex multiplication in this conductor range. Every rank 2 curve in this range has $\#\Sha(\Q,E)_\text{an} = 1$. For any curve $E$ in this range, $\ord_p(\Sha(\Q,E)_\text{an}) \leq 6$ for all primes $p$. If such an $E$ is optimal then $\ord_p(\Sha(\Q,E)_\text{an}) \leq 4$ for all primes $p$.


\begin{thm}\label{p_2}
If $E/\Q$ has conductor $N < 5000$, then $\BSD(E, 2)$ is true.
\end{thm}
\begin{proof}Assume that $E$ is an optimal curve and let $T(E) = \ord_2(\#\Sha(\Q,E)_\text{an})$. For each curve we are considering, if $T(E) = 0$ then a 2-descent proves $\BSD(E, 2)$ and if $T(E) > 0$ then a 2-descent proves $\Sha(\Q,E)[2] \cong (\Z/2\Z)^2$. If $T(E) = 2$ then a 4-descent proves $\BSD(E, 2)$ and if $T(E) > 2$ then a 4-descent proves $\Sha(\Q,E)[4] \cong (\Z/4\Z)^2$. For the range of curves we are considering $T(E)$ is at most 4 and if $T(E) = 4$, an 8-descent proves that $\Sha(\Q,E)[8] = \Sha(\Q,E)[4]$ and hence proves $\BSD(E, 2)$.
\end{proof}

\begin{thm}\label{p_3}
If $E/\Q$ has conductor $N < 5000$, then $\BSD(E, 3)$ is true.
\end{thm}
\begin{proof}For optimal curves where $\ord_3(\#\Sha(\Q,E)_\text{an})$ is trivial, a 3-descent suffices. For the rest we have that $\ord_3(\#\Sha(\Q,E)_\text{an}) = 2$, and in this case a 3-descent proves that $\Sha(\Q,E)[3] \cong (\Z/3\Z)^2$. These 31 remaining optimal curves are shown in Table \ref{nontriv_3sha}. If $E$ is in the set \{681b1, 1913b1, 2006e1, 2429b1, 2534e1, 2534f1, 2541d1, 2674b1, 2710c1, 2768c1, 2849a1, 2955b1, 3054a1, 3306b1, 3536h1, 3712j1, 3954c1, 4229a1, 4592f1, 4606b1\}, then the algorithm of Stein and Wuthrich \cite{shark} proves the desired upper bound. For the rest of the curves except for 2366d1 and 4914n1, the mod-3 representations are surjective. Table \ref{nontriv_3sha_heegner} displays selected Heegner indexes in this case, which together with Theorem \ref{big-result} (and Theorem \ref{jetchev} for 4675j1 since $c_{17}(\text{4675j1}) = 3$) proves the desired upper bound.

\begin{table}
 \caption{Optimal $E$ with $\ord_3(\#\Sha(\Q,E)_\text{an}) = 2$}
 \label{nontriv_3sha}
\begin{center}  \begin{tabular}{|l|l|l|l|l|l|l|l|} \hline
681b1 & 2429b1 & 2601h1 & 2768c1 & 3054a1 & 3712j1 & 4229a1 & 4675j1\\
1913b1 & 2534e1 & 2674b1 & 2849a1 & 3306b1 & 3879e1 & 4343b1 & 4914n1\\
2006e1 & 2534f1 & 2710c1 & 2932a1 & 3536h1 & 3933a1 & 4592f1 & 4963c1\\
2366d1 & 2541d1 & 2718d1 & 2955b1 & 3555e1 & 3954c1 & 4606b1 & \\\hline
 \end{tabular}\end{center}
 \end{table} 

\begin{table}
 \caption{Heegner indices where $\ord_3(\#\Sha(\Q,E)_\text{an}) = 2$}
 \label{nontriv_3sha_heegner}
\begin{center}
\begin{tabular}{cc}
  \begin{tabular}{|l|r|r|r|} \hline
$E$ &      $D$ & $I_D$ & $\ord_3(I_D)$\\\hline
2601h1    &-8       &12&1\\
2718d1    &-119    &48&1\\
2932a1    &-31     &3&1\\
3555e1    &-56     &6&1\\
3879e1    &-35     &24&1\\\hline
\end{tabular}
 &
  \begin{tabular}{|l|r|r|r|} \hline
$E$ &      $D$ & $I_D$ & $\ord_3(I_D)$\\\hline
3933a1    &-56     &24&1\\
4343b1    &-19     &12&1\\
4675j1     &-19     &18&2\\
4963c1    &-19     &3&1\\&&&\\\hline
 \end{tabular}
 \end{tabular}
  \end{center}
 \end{table} 

Finally we are left with 2366d1 and 4914n1. Each isogeny class contains a curve $F$ for which $\#\Sha(\Q,F)_\text{an}=1$, so we replace these curves with 2366d2 and 4914n2. Then 3-descent shows that $\Sha(\Q,F)[3] = 0$, and hence $\BSD(F, 3)$ for both curves.
\end{proof}

\begin{cor}If $\rank(E(\Q)) = 0$, $E$ has conductor $N < 5000$ and $E$ has complex multiplication, then the full BSD conjecture is true.
\end{cor}
\begin{proof}
 This is a direct result of Corollary \ref{CM_zero} and Theorems \ref{p_2} and \ref{p_3}.
\end{proof}

\begin{thm}\label{opt_sha_nontriv}
If $E/\Q$ is an optimal curve with conductor $N < 5000$ and nontrivial analytic $\Sha$, i.e. $\#\Sha(\Q,E)_\text{an} \neq 1$, then for every $p \mid \#\Sha(\Q,E)_\text{an}$, $\BSD(E, p)$ is true.
\end{thm}
\begin{proof}
By \cite{crem-130K} we have that $p \leq 7$, and by the theorems of the previous section, we may assume that $p \geq 5$.

For $p = 5$, $E$ is one of the twelve curves listed in Table \ref{nontriv_5sha}. These are all rank 0 curves with $E[5]$ surjective, so if $5 \nmid N$ Theorem \ref{kato} provides an upper bound of 2 for $\ord_5(\#\Sha(\Q,E))$. This leaves just 2900d1 and 3185c1. For 2900d1, Corollary \ref{bounding_heeg_ind} together with a point search shows that the Heegner index is at most 23 for discriminant -71, hence Kolyvagin's inequality provides the upper bound of 2 in this case. For 3185c1, the algorithm of Stein and Wuthrich \cite{shark} provides the upper bound of 2. In all twelve cases \cite{crem-maz-vis} (and the appendix of \cite{agashe-stein-vis}) finds visible nontrivial parts of $\Sha(\Q,E)[5]$. Since the order must be a square, $\#\Sha(\Q,E)$ must be exactly 25 in each case.

\begin{table}
 \caption{Optimal $E$ with $\ord_5(\#\Sha(\Q,E)_\text{an}) = 2$}
 \label{nontriv_5sha}
\begin{center}  \begin{tabular}{|l|l|l|l|l|l|} \hline
1058d1 & 1664k1 & 2574d1 & 2900d1 & 3384a1 & 4092a1\\
1246b1 & 2366f1 & 2834d1 & 3185c1 & 3952c1 & 4592d1\\\hline
 \end{tabular}\end{center}
 \end{table} 

For $p = 7$ there is only one curve $E$ = 3364c1 and $E[7]$ is surjective. Since $7 \nmid 3364$ and $E$ is a rank 0 curve without complex multiplication, Theorem \ref{kato} bounds $\ord_7(\#\Sha(\Q,E))$ from above by 2. Furthermore, Grigorov's thesis \cite[p. 88]{grigorov_phd} shows that $\ord_7(\#\Sha(\Q,E))$ is bounded from below by 2. Alternatively, the elements of $\Sha(\Q,E)[7]$ are visible at three times the level, as Tom Fisher kindly pointed out -- one should also be aware of his tables of nontrivial elements of $\Sha$ of order three and five, available on his website\footnote{\url{http://www.dpmms.cam.ac.uk/~taf1000/}}.
\end{proof}


\begin{thm}\label{rank0_irr}
If $E/\Q$ is an optimal rank 0 curve with conductor $N < 5000$ and $p$ is a prime such that $E[p]$ is irreducible, then $\BSD(E, p)$ is true.
\end{thm}
\begin{proof}
By theorems of the previous two sections, we may assume that $p > 3$, $E$ does not have complex multiplication and $\ord_p(\#\Sha(\Q,E)_\text{an}) = 0$. In this case Theorem \ref{kato} applies to $E$ (since the rank part of the conjecture is known for $N < 130000$ by \cite{crem-130K}). At first, suppose that $E$ does not have additive reduction at $p$.

Suppose that $E[p]$ is surjective. In this case we need only consider primes dividing the conductor $N$. For such pairs $(E,p)$, we can compute the Heegner index or an upper bound for it, which gives an upper bound on $\ord_p(\Sha(\Q,E))$. When the results of Kolyvagin and Jetchev were not strong enough to prove $\BSD(E, p)$ using the first available Heegner discriminant, the algorithm of Stein and Wuthrich \cite{shark} was (although to be fair the former may be strong enough using other Heegner discriminants in these cases). This algorithm always provides a bound in this situation since $p > 3$ is a surjective prime of non-additive reduction and $E$ is rank 0.

Now suppose that $E[p]$ is not surjective. The curve-prime pairs matching these hypotheses can be found in Table \ref{rank0_irr_nonsur_nonadd} along with selected Heegner indices. The only prime to occur in these pairs is 5, and each chosen Heegner discriminant and index is not divisible by 5 except for $E$ = 3468h. Further, 5 does not divide the conductor of any of these curves so by Cha's theorem \ref{cha}, $\BSD(E, 5)$ is true for these pairs. For $E$ = 3468h note that one of the Tamagawa numbers is 5, so by Theorem \ref{jetchev}, $\BSD(E, 5)$ is true for this curve.

\begin{table}
\caption{Non-additive reduction, irreducible but not surjective}
\label{rank0_irr_nonsur_nonadd}
\begin{center}  \begin{tabular}{|lrrr|lrrr|lrrr|} \hline
$E$ & $p$ & $D$ & $I_D$ & $E$ & $p$ & $D$ & $I_D$ & $E$ & $p$ & $D$ & $I_D$\\\hline
324b1 & 5 & -23 & 6     & 1296g1 & 5 & -23 & 2            & 3468c1 & 5 & -47 & 2\\
324d1 & 5 & -23 & 2     & 1296i1 & 5  & -23 & 2            & 3468h1 & 5 & -47 & $\leq 11$\\
608b1 & 5 & -31 & 2     & 1444a1 & 5 & -31 & 2            & 4176n1 & 5 & -23 & $\leq 3$\\
648c1 & 5 & -23 & 4     & 2268a1 & 5 & -47 & 6            & 4232b1 & 5 & -7 & 2\\
1044a1 & 5 & -23 & 12 & 2268b1 & 5 & -47 & $\leq 3$ & 4232d1 & 5 & -7 & 6\\
1216i1 & 5  & -31 & 1   & 3132a1 & 5 & -23 & 6            & &&&\\
\hline
\end{tabular}\end{center}
\end{table}

We are now left to consider the 1964 pairs $(E,p)$ for which $E$ has additive reduction at $p$. There are 14 pairs where $E[p]$ is not surjective, and Theorem \ref{stein-et-al} applies to all of them. The Heegner point height calculations listed in Table \ref{add_irr_not_surj} prove that $\BSD(E, p)$ is true in these cases. Note that when $p$ may divide the Heegner index, it must do so of order at most 1, and in these cases it also divides a Tamagawa number, so Theorem \ref{jetchev} assists Theorem \ref{stein-et-al}.

\begin{table}
\caption{Additive reduction, irreducible but not surjective}
\label{add_irr_not_surj}
\begin{center}  \begin{tabular}{|lrrrr|lrrrr|} \hline
$E$ & $p$ & $D$ & $I_D$ & $\tau_p$ & $E$ & $p$ & $D$ & $I_D$ & $\tau_p$\\\hline
675d1 & 5 & -11 &   2& 1 & 2400bg1 & 5& -71 &  20& 10\\
675f1 & 5 & -11  &  2& 1 & 2450d1 &7  &-31  &  1& 1\\
800e1 & 5 & -31  &  6& 3 & 2450bd1 &7 &-31 & $< 13$& 7\\
800f1 & 5 & -31  &  2& 1 & 4800n1 &5  &-71 & $< 5$& 3\\
1600i1 & 5 & -31 &   4& 2 & 4800u1  &5 &-71 & 10& 5\\
1600k1 & 5 & -31  &  4& 2 & 4900s1 &5 & -31 &   4& 2\\
2400f1 &5 &-191 & $< 5$& 2 & 4900u1 &5&  -31 &  12& 6\\
\hline
\end{tabular}\end{center}
$\tau_p = \ord_p(\prod_qc_q)$.
\end{table}

Now we may also assume that $E[p]$ is surjective. In these cases, Heegner index computations sufficed to prove $\BSD(E, p)$, using Theorem \ref{big-result} and Theorem \ref{jetchev}. For 79 of these curves the Heegner index computation required 4- and even 6-descent \cite{tom_fisher_6_12}. These are listed in Table \ref{add_surj_hard}, thanks to Tom Fisher.
\end{proof}

For example if $E = \text{1050c1}$, the first available Heegner discriminant is -311. Bounding the Heegner index is difficult in this case since it involves point searches of prohibitive height. However in two and a half seconds the algorithm of Stein and Wuthrich provides an upper bound of 0 for the 7-primary part of the Shafarevich-Tate group, which eliminates the last prime for that curve.

\begin{table}
\caption{Additive reduction, surjective}
\label{add_surj_hard}
\begin{center}
\begin{tabular}{|lrrrr|} \hline
     $E$ &$p$ &   $D$ & $I_D$ & $\tau_p$\\\hline
  1050l1 &  5 &  -311 &     3 & 0 \\
  1050n1 &  5 & -2399 &    19 & 0 \\
  1050q1 &  5 &  -311 &     7 & 0 \\
  1350o1 &  5 &  -239 &     4 & 0 \\
  1470q1 &  7 &  -479 &    26 & 0 \\
  1764h1 &  7 &  -167 &     6 & 0 \\
  1850d1 &  5 &  -471 &     6 & 0 \\
  2100o1 &  5 &  -311 &     4 & 0 \\
  2352x1 &  7 &  -551 &     6 & 0 \\
 2450bd1 &  5 &  -559 &    14 & 0 \\
  2450k1 &  5 &  -159 &     2 & 0 \\
 2550bc1 &  5 &  -191 &     7 & 0 \\
  2550j1 &  5 &  -239 &    23 & 0 \\
  2550z1 &  5 & -1511 &    45 & 1 \\
 2646ba1 &  7 &   -47 &    11 & 0 \\
 2646bd1 &  7 &  -143 &    10 & 0 \\
  2650k1 &  5 &  -679 &    28 & 0 \\
  3038m1 &  7 &   -55 &     6 & 0 \\
 3150bc1 &  5 & -1511 &     6 & 0 \\
 3150bd1 &  5 & -1991 &    64 & 0 \\
 3150bj1 &  5 &  -311 &     2 & 0 \\
 3150bn1 &  5 & -1991 &    22 & 0 \\
  3150t1 &  5 & -1151 &     6 & 0 \\
  3185c1 &  7 &  -199 &    10 & 0 \\
  3225b1 &  5 &  -119 &     4 & 0 \\
  3234c1 &  7 &  -503 &    16 & 0 \\
  3350d1 &  5 &   -79 &    12 & 0 \\
  3450p1 &  5 &  -479 &    13 & 0 \\
  3450v1 &  5 &  -191 &   180 & 1 \\
  3630c1 & 11 & -1559 &     4 & 0 \\
  3630l1 & 11 &  -239 &    35 & 0 \\
  3630r1 & 11 &  -239 &     7 & 0 \\
  3630u1 & 11 & -1319 &     9 & 0 \\
  3650j1 &  5 &   -79 &    14 & 0 \\
 3822bc1 &  7 &  -647 &    18 & 0 \\
  3822e1 &  7 & -1511 &     2 & 0 \\
  3822u1 &  7 &  -503 &    10 & 0 \\
  3822w1 &  7 &  -503 &     6 & 0 \\
  3822z1 &  7 & -1823 &    32 & 0 \\
  3850e1 &  5 & -1399 &     6 & 0 \\
\hline
\end{tabular}
\begin{tabular}{|lrrrr|} \hline
     $E$ &$p$ &   $D$ & $I_D$ & $\tau_p$\\\hline
  3850m1 &  5 & -2351 &     2 & 0 \\
  3850y1 &  5 & -1399 &    54 & 0 \\
  3900k1 &  5 & -1199 &     4 & 0 \\
  3900l1 &  5 &  -191 &    30 & 1 \\
 4050bi1 &  5 &   -71 &     4 & 0 \\
  4050s1 &  5 &  -551 &     6 & 0 \\
  4050x1 &  5 &  -119 &     6 & 0 \\
 4200bd1 &  5 &  -479 &    27 & 0 \\
  4200m1 &  5 &  -719 &    32 & 0 \\
  4350q1 &  5 &  -719 &    11 & 0 \\
  4350w1 &  5 &  -719 &    24 & 0 \\
  4410b1 &  7 &  -671 &     4 & 0 \\
 4410bi1 &  7 & -1319 &    18 & 0 \\
 4410bj1 &  7 &  -311 &     6 & 0 \\
  4410q1 &  7 &  -839 &     4 & 0 \\
  4410u1 &  7 & -2231 &    10 & 0 \\
  4550p1 &  5 & -1119 &    14 & 0 \\
  4606b1 &  7 &   -31 &    12 & 0 \\
 4650bo1 &  5 &  -119 &    18 & 0 \\
 4650bs1 &  5 &  -239 &    84 & 0 \\
 4650bt1 &  5 & -1511 &     2 & 0 \\
 4650bu1 &  5 & -1199 &   170 & 1 \\
  4650q1 &  5 &  -119 &     6 & 0 \\
  4650w1 &  5 &  -719 &    46 & 0 \\
  4725q1 &  5 &   -59 &     8 & 0 \\
 4800ba1 &  5 &   -71 &     7 & 0 \\
  4850h1 &  5 &   -31 &    22 & 0 \\
  4900w1 &  5 &  -311 &     8 & 0 \\
 4950bj1 &  5 &  -239 &     6 & 0 \\
 4950bk1 &  5 &  -239 &    14 & 0 \\
 4950bm1 &  5 &  -479 &    56 & 0 \\
 4950bp1 &  5 &  -431 &    22 & 0 \\
  4950w1 &  5 & -1151 &     4 & 0 \\
  4950x1 &  5 &  -359 &    12 & 0 \\
 4998bg1 &  7 &   -47 &    18 & 0 \\
 4998bk1 &  7 &   -47 &    36 & 0 \\
  4998k1 &  7 &   -47 &    30 & 0 \\
  4998t1 &  7 & -1487 &     6 & 0 \\
  4998u1 &  7 &   -47 &     6 & 0 \\
&&&&\\ \hline
\end{tabular}
\end{center}
$\tau_p = \ord_p(\prod_qc_q)$.
\end{table}



\begin{prop}\label{1155k}
If $E$ is the elliptic curve 1155k and $p = 7$, then $\BSD(E, p)$ is true.
\end{prop}Note that for $(E,p) = (\text{1155k},7)$, we have $c_3(E) = 7, c_5(E) = 7,$
$$
\ord_7(\#\Sha(\Q,E)_\text{an}) = 0 \text{\ \ and\ \ } \ord_7(\#\Sha(\Q,E)) \leq 2,
$$
by Theorem \ref{jetchev}.
The following proof is due to Christian Wuthrich.
\begin{proof}

First, note that $E/\Q$ has non-split multiplicative reduction at 7. Let $D = -8$ and let $K = \Q(\sqrt{D})$, noting that $E(K) = E(\Q) \cong \Z$ and that $\#E^D(\Q) = 1$. Since 7 is inert in $K$, the reduction of $E/K$ at $7\O_K$ is split multiplicative. Kato's theorem \cite[Thm. 17.4]{kato04} is known to hold for curves with multiplicative reduction over abelian fields unramified at $p$. The characteristic series $f(T)$ of the dual of the Selmer group therefore divides the $p$-adic $L$-series
$$
L_p(E/K,T) = L(E/\Q,T) \cdot L(E^D/\Q,T).
$$

By work of Jones \cite[Thm. 3.1]{jones_mul}, we can compute the order of vanishing of $f(T)$ at $T=0$, which is 2 since the reduction is split multiplicative, and the leading term which is, up to a unit in $\Z_p^\times$,
$$
\frac{\prod_v c_v \cdot \#\tilde E(\F_{49}) \cdot \#\Sha(K, E)[7^\infty] \cdot \Reg_p(E(K)) \cdot \mathcal{L}}{\#E(K)_\tors^2},
$$
where $\mathcal{L}$ is the $L$-invariant and $\Reg_p$ is the $p$-adic regulator as defined in the split multiplicative case by \cite{MTT86} and corrected by \cite{werner_split_mult}.

We compute $\prod_v c_v = 7^3$, $\tilde E(\F_{49}) \cong \Z/48\Z$ and
\begin{align*}
 L_p(E/\Q, T)   &= (6\cdot 7 + O(7^2))\cdot T + (4\cdot 7 + O(7^2))\cdot T^2 + O(T^3)\\
 L_p(E^D/\Q, T) &= (2\cdot 7 + O(7^2))\cdot T + (4\cdot 7 + O(7^2))\cdot T^2 + O(T^3).
\end{align*}

To compute the $L$-invariant $\mathcal L$ we switch to the Tate curve. Since $E^D/\Q$ has split multiplicative reduction at 7 and the parameter is the same as for $E/K$, we have
$$
q_E = 3 \cdot 7 + 3 \cdot 7^2 + 4 \cdot 7^3 + 7^5 + O(7^6).
$$
Hence the $L$-invariant is
$$
\mathcal L = \log_p(q_E)/\ord_p(q_E) = 2\cdot 7 + 6 \cdot 7^2 + 2 \cdot 7^3 + 5 \cdot 7^4 + O(7^6).
$$

Finally we wish to compute the $p$-adic regulator. If $P$ is a generator of $E(K)$, then $Q = 7\cdot 8 \cdot P$ has good reduction everywhere
 and lies in the formal group at the place $7 \O_K$.
 One computes, as in \cite[\S 4.2]{shark}, the $p$-adic height of $Q$ and so that of $P$:
$$
h_p(P) = \frac{h_p(Q)}{(7\cdot 8)^2} = 2\cdot 7^{-1} + 4 + 5 \cdot 7 + 2 \cdot 7^2 + 7^3 + 3\cdot 7^5 + O(7^6).
$$

Since the leading term of the $p$-adic $L$-function is $5\cdot 7^2 + O(7^3)$ and the leading term of $f(T)$ must have smaller valuation, we have
$$
\ord_p\left(7^3 \cdot 48 \cdot \#\Sha(K, E)[7^\infty] \cdot \frac{h_p(P)}7 \cdot \mathcal L\right) \leq 2.
$$
Therefore,
$$
\ord_p(\#\Sha(K, E)[7^\infty]) \leq -\ord_p(h_p(P)) - \ord_p(\mathcal L) = 0.
$$
In particular, $\ord_7(\#\Sha(\Q,E)) = 0$.
\end{proof}
It may also be possible to prove this using \cite{greenberg-vatsal}, since there is a modular congruence 77a[7] $\cong$ 1155k[7].

\begin{thm}\label{rank_1_irr}
If $E/\Q$ is a rank 1 curve with conductor $N < 5000$ and $p$ is a prime such that $E[p]$ is irreducible, then $\BSD(E, p)$ is true.
\end{thm}
\begin{proof}
We may assume in addition that $E$ is optimal, since reducibility is isogeny-invariant. By Theorems \ref{p_2} and \ref{p_3}, if $p < 5$ then $\BSD(E, p)$ is true. Thus we may assume $p \geq 5$. Computing the Heegner index is much easier when $E$ has rank 1, as noted in Section \ref{heeg_pts}. Kolyvagin's theorem then rules out many pairs $(E,p)$ right away. Then some combination of Theorems \ref{stein-et-al}, \ref{cha}, \ref{jetchev} and the algorithm of Stein and Wuthrich \cite{shark} will rule out many more pairs. If no combination of these techniques works for the first Heegner index one computes, then another Heegner discriminant must be used. Table \ref{heeg_disc_rank_1_5_and_7} lists rank 1 curves $E$ for which this is necessary, such that $E[p]$ is irreducible, $E$ does not have complex multiplication and $(E,p) \neq (\text{1155k},7)$. All these curves have $E[p]$ surjective and $p$ does not divide any Tamagawa numbers so it is sufficient to demonstrate a Heegner index which $p$ does not divide. The case $(\text{1155k},7)$ is Proposition \ref{1155k}.

\begin{table}
\caption{Some Heegner indexes using larger discriminants}
\label{heeg_disc_rank_1_5_and_7}
\begin{center}  \begin{tabular}{|lrrr|lrrr|lrrr|} \hline
$E$ & $p$ & $D$ & $I_D$ & $E$ & $p$ & $D$ & $I_D$ & $E$ & $p$ & $D$ & $I_D$\\\hline
1450c1 & 5 & -151 & 3   & 3150i1 & 5 & -479 & 8    & 4440f1 & 5 & -259 & 2\\
1485e1 & 5 & -131 & 4   & 3150bb1 & 5 & -479 & 4 & 4485d1 & 5 & -296 & 2\\
1495a1 & 5 & -79 & 3     & 3310b1 & 5 & -151 & 3   & 4550j1 & 5 & -199 & 4\\
1735a1 & 5 & -24 & 4     & 3450b1 & 5 & -551 & 28 & 4675t1 & 5 & -84 & 9\\
2090c1 & 5 & -431 & 8   & 3480h1 & 5 & -239 & 2   & 4680h1 & 5 & -311 & 8\\
2145a1 & 5 & -131 & 2   & 3630h1 & 5 & -431 & 3   & 4725c1 & 5 & -104 & 8\\
2275b1 & 5 & -139 & 2   & 3760k1 & 5 & -39 & 1     & 4800bx1 & 5 & -119 & 7\\
2550n1 & 5 & -239 & 9   & 3900n1 & 5 & -599 & 2   & 4815e1 & 5 & -71 & 6\\
2860a1 & 5 & -519 & 9   & 3920y1 & 5 & -159 & 6   & 4950r1 & 5 & -359 & 6\\
2970j1 & 5 & -359 & 3    & 4050h1 & 5 & -239 & 32 &  &&&\\
2990e1 & 5 & -159 & 12 & 4140c1 & 5 & -359 & 6   & 2660a1 & 7 & -439 & 11\\
3060h1 & 5 & -359 & 18 & 4200t1 & 5 & -551 & 4    & 4158a1 & 7 & -215 & 2\\
3075a1 & 5 & -119 & 14 & 4400z1 & 5 & -79 & 24   & 4704t1 & 7 & -143 & 8\\
3140b1 & 5 & -39 & 2     & 4410i1 & 5 & -479 & 2    & 4914x1 & 7 & -335 & 12\\
\hline
\end{tabular}\end{center}
\end{table}

If $E$ has complex multiplication, there are 17 pairs $(E,p)$ left, namely the six pairs in Example \ref{CM_rk1_5}, the four in Example \ref{CM_rk1_7}, the five in Example \ref{CM_rk1_11}, and the two in Example \ref{CM_rk1_5descent}. Table \ref{CM_rank_1_heeg_discs} lists Heegner indexes which, together with Theorem \ref{big-result}, prove the fifteen cases not handled in Example \ref{CM_rk1_5descent}.
\end{proof}

\begin{table}[h]
\caption{Heegner indexes of some rank 1 curves with complex multiplication}
\label{CM_rank_1_heeg_discs}
\begin{center}
\begin{tabular}{|lrrr|lrrr|} \hline
$E$   &$p$& $D$ &$I_D$& $E$   &$p$& $D$ & $I_D$\\\hline
675a  & 5 & -11 &  2  & 3136v & 7 & -47 & 2 \\
900c  & 5 & -119&  6  & 3267d & 11& -8  & 2 \\
1568g & 7 & -31 &  2  & 3600bd& 5 & -71 & 12\\
2700h & 5 & -119&  3  & 3872a & 11& -7  & 2 \\
2700l & 5 & -119&  3  & 4356a & 11& -95 & 4 \\
2700p & 5 & -71 &  6  & 4356b & 11& -167& 6 \\
3136t & 7 & -55 &  2  & 4356c & 11& -95 & 2 \\
3136u & 7 & -31 &  4  & &&&\\
\hline
\end{tabular}
\end{center}
\end{table}

\section{Curves of conductor $N < 5000$, reducible mod-$p$ representations}\label{reducible}
Suppose $E$ is an optimal elliptic curve of conductor $N < 5000$ and $p$ is a prime such that $E[p]$ is reducible, i.e., there is a $p$-isogeny $\phi : E \rightarrow E^\prime$. If $p < 5$ or $E$ is a rank 0 curve with complex multiplication, results of the previous sections show that $\BSD(E, p)$ is true. This leaves 464 pairs $(E,p)$. By results in \cite{mazur-eisenstein}, $\BSD(\mathrm{11a}, 5)$ is true, leaving 463 pairs. The results of Theorem \ref{stein-et-al} can be applied to 339 of these curve-prime pairs, using Corollary \ref{bounding_heeg_ind} and various descents, including \cite{tom_fisher_6_12}. This leaves 124 pairs of the original 464: 103 5-isogenies, 16 7-isogenies, 2 11-isogenies, and one isogeny each of degree 19, 43 and 67. There are also two cases with rank 2, namely $(E,p) \in \{(\text{2601l}, 5), (\text{3328d}, 5)\}$. Of the 122 rank 0 and 1 cases remaining, 103 more at $p = 5$ and $p = 7$ are covered in \cite{tom-fisher-phd}.

Of the nineteen remaining cases, eight are proven in a paper by Michael Stoll and the author \cite{isogeny-descent}. The eleven remaining are listed in Table \ref{isogenies_left}: if $(E,p)$ does not appear in Table \ref{isogenies_left} for $E[p]$ reducible, then $\BSD(E, p)$ is true. For 5-isogenies involving 5-torsion in $\Sha$ (eight cases), one curve has trivial $\Sha$ in each case, and a full 5-descent on that curve involves number fields of degree at most 20. For 7-isogenies involving 7-torsion in $\Sha$ (three cases), similarly one curve has trivial $\Sha$ but a full 7-descent on that curve involves number fields of degree 28. Extending a 7-isogeny descent to a second descent and thus obtaining a full 7-descent would resolve these last three cases.

\begin{table}
\caption{Remaining curves: reducible representations}
\label{isogenies_left}
\begin{center}
\begin{tabular}{|lr|lr|}\hline
$E$  & $p$ & $E$   & $p$ \\\hline
546f &  7  & 1938j &  5 \\
570l &  5  & 1950y &  5 \\
858k &  7  & 2550be&  5 \\
870i &  5  & 2370m &  5 \\
1050o&  5  & 3270h &  5 \\
1230k&  7  & &\\
\hline
\end{tabular}
\end{center}
\end{table}

\bibliographystyle{amsplain}
\bibliography{/home/rlmill/Documents/math/Uber.bib}

\end{document}